\documentclass[10]{article}
\usepackage{amsfonts,amssymb,amsmath,amsthm,cite}
\usepackage{graphicx}

\usepackage[applemac]{inputenc}
\usepackage{amsmath,amssymb,amsthm, hyperref, euscript}
\usepackage[matrix,arrow,curve]{xy}
\usepackage{graphicx}
\usepackage{tabularx}
\usepackage{float}
\usepackage{hyperref}
\usepackage{tikz}
\usepackage{slashed}
\usepackage{mathrsfs}

\usetikzlibrary{matrix}

\usepackage[T1]{fontenc}
\usepackage{amsfonts,cite}
\usepackage{graphicx}

\usepackage[applemac]{inputenc}

\usepackage[sc]{mathpazo}
\DeclareFontFamily{T1}{pzc}{}
\DeclareFontShape{T1}{pzc}{m}{it}{1.8 <-> pzcmi8t}{}
\DeclareMathAlphabet{\mathpzc}{T1}{pzc}{m}{it}

\theoremstyle{plain}
\newtheorem{prop}{Proposition}[section]

\newtheorem{lem}[prop]{Lemma}
\newtheorem{cor}[prop]{Corollary}
\newtheorem{thm}[prop]{Theorem}
\newtheorem{theorem}[prop]{Theorem}

\theoremstyle{definition}
\newtheorem{defn}[prop]{Definition}
\newtheorem{rem}[prop]{Remark}

\theoremstyle{definition}
\newtheorem{definition}[prop]{Definition}

\newtheorem{remark}[prop]{Remark}
\numberwithin{equation}{section}

\newcommand{\vertiii}[1]{{\left\vert\kern-0.25ex\left\vert\kern-0.25ex\left\vert #1
    \right\vert\kern-0.25ex\right\vert\kern-0.25ex\right\vert}}

\usepackage[sc]{mathpazo}
\newbox\ncintdbox \newbox\ncinttbox 
\setbox0=\hbox{$-$} \setbox2=\hbox{$\displaystyle\int$}
\setbox\ncintdbox=\hbox{\rlap{\hbox
    to \wd2{\hskip-.125em \box2\relax\hfil}}\box0\kern.1em}
\setbox0=\hbox{$\vcenter{\hrule width 4pt}$}
\setbox2=\hbox{$\textstyle\int$} \setbox\ncinttbox=\hbox{\rlap{\hbox
    to \wd2{\hskip-.175em \box2\relax\hfil}}\box0\kern.1em}

\newcommand{\blank}{-}           

\newcommand{\C}{\mathbb{C}}                  
\newcommand{\G}{\mathcal{G}}                 
\newcommand{\hookto}{\hookrightarrow}        
\newcommand{\La}{\Lambda}                    
\newcommand{\la}{\lambda}                    
\def\<#1|#2>{\langle#1\stroke#2\rangle} 
\def\?#1|#2?{\{#1\stroke#2\}}        

\def\<#1,#2>{\langle#1,#2\rangle}            
\def\ee_#1{e_{{\scriptscriptstyle#1}}}       
\def\wick:#1:{\mathopen:#1\mathclose:}       

\newbox\ncintdbox \newbox\ncinttbox 
\setbox0=\hbox{$-$}
\setbox2=\hbox{$\displaystyle\int$}
\setbox\ncintdbox=\hbox{\rlap{\hbox
           to \wd2{\box2\relax\hfil}}\box0\kern.1em}
\setbox0=\hbox{$\vcenter{\hrule width 4pt}$}
\setbox2=\hbox{$\textstyle\int$}
\setbox\ncinttbox=\hbox{\rlap{\hbox
           to \wd2{\hskip-.05em\box2\relax\hfil}}\box0\kern.1em}

\newcommand{\stroke}{\mathbin|}   

\newcommand{\be}{\begin{equation}}
\renewcommand{\ee}{\end{equation}}
\newcommand{\bea}{\begin{eqnarray}}
\newcommand{\eea}{\end{eqnarray}}
\newcommand{\bean}{\begin{eqnarray*}}
	\newcommand{\eean}{\end{eqnarray*}}
\newcommand{\brray}{\begin{array}}
	\newcommand{\erray}{\end{array}}

\title{Quantization of Noncompact Coverings and Strong Morita Equivalence}
\begin{document}
\maketitle  \setlength{\parindent}{0pt}
\begin{center}
\author{
{\textbf{Petr R. Ivankov*}\\
e-mail: * monster.ivankov@gmail.com \\
}
}
\end{center}

\vspace{1 in}

\noindent

\paragraph{}

Any finite algebraic Galois covering corresponds to an algebraic Morita equivalence. Here the $C^*$-algebraic analog of this fact is proven, i.e. any noncommutative finite-fold covering corresponds to a strong Morita equivalence.  

\section{Motivation. Preliminaries}
\paragraph*{}
It is known that any finite algebraic Galois covering corresponds to an algebraic Morita equivalence (cf. \cite{auslander:galois,mont:hopf-morita}). However in case of $C^*$-algebras the strong Morita equivalence yields a more adequate picture then pure algebraic Morita equivalence, e.g. strong Morita equivalence can be applied for nonunital $C^*$-algebras.
This article describes an association between  finite-fold noncommutative coverings and strong Morita equivalences.




  


\subsection{Morita Equivalence}

\subsubsection{Algebraic Morita Equivalence}

\begin{defn}\label{morita_ctx_defn}
A \textit{Morita context} $\left( A,B,P,Q,\varphi,\psi\right)$ or, in some authors (e.g. Bass \cite{bass}) the \textit{pre-equivalence data} is a generalization of Morita equivalence between categories of modules. In the case of right modules, for two associative $\mathbf{k}$-algebras (or, in the case of $\mathbf{k} = \mathbb{Z}$, rings) $A$ and $B$, it consists of bimodules $_AP_B$, $_BQ_A$ and bimodule homomorphisms $\varphi: P\otimes_B Q\to A$, $\psi: Q\otimes_A P\to B$ satisfying mixed associativity conditions, i.e. for any $p,p' \in P$ and $q,q' \in Q$ following conditions hold:
\begin{equation}\label{morita_ctx_eqn}
\begin{split}
\varphi\left(p\otimes q \right) p' = p \psi\left(q\otimes p' \right),\\
\psi\left(q \otimes p \right) q' = q\varphi\left(p\otimes q' \right).  
\end{split}
\end{equation}
A Morita context is a \textit{Morita equivalence} if both $\varphi$ and $\psi$ are isomorphisms of bimodules. 
\end{defn}

\begin{rem}\label{morita_rem}
The Morita context $\left( A,B,P,Q,\varphi,\psi\right)$  is a Morita equivalence if and only if $A$-module $P$ is a finitely generated projective generator (cf. \cite{bass} II 4.4)
\end{rem}


\subsubsection{Strong Morita equivalence for $C^*$-algebras}

\begin{definition}[Paschke \cite{Paschke:73}, Rieffel \cite{Rieffel:74a}]
	Let~$A$ be a $C^*$-algebra.  A \emph{pre-Hilbert $A$-module} is a right
	$B$-module~$X$ (with a compatible $\C$-vector space structure),
	equipped with a conjugate-bilinear map (linear in the second variable)
	$\left\langle{\blank},{\blank}\right\rangle_A\colon X\times X\to A$ satisfying
	\begin{enumerate}
		\item[(a)] $\left\langle{x},{x}\right\rangle_A\ge0$ for all $x\in X$;
		\item[(b)] $\left\langle{x},{x}\right\rangle_A=0$ only if $x=0$;
		\item[(c)] $\left\langle{x},{y}\right\rangle_A=\left\langle{y},{x}\right\rangle_A^\ast$ for all $x,y\in X$;
		\item[(d)] $\left\langle{x},{y\cdot a}\right\rangle_A=\left\langle{x},{y}\right\rangle_B\cdot a$ for all $x,y\in X$, $a\in A$.
	\end{enumerate}
	The map $\left\langle{\blank},{\blank}\right\rangle_A$ is called a \emph{$A$-valued inner product
		on~$X$}. 
\end{definition}

It can be shown that $\|x\|=\|\left\langle{x},{x}\right\rangle_A\|^{1/2}$ defines a norm on~$X$.
If~$X$ is complete with respect to this norm, it is called a \emph{Hilbert
	$A$-module}.  

\begin{definition}\label{strong_morita_defn}[Rieffel \cite{Rieffel:74a}, \cite{Rieffel:76}]
	Let $A$ and~$B$ be $C^*$-algebras.  By an \emph{$B$-$A$-equivalence
		bimodule} we mean an $\left(B,A\right)$-bimodule which is equipped with $A$- and
	$B$-valued inner products with respect to which~$X$ is a right Hilbert
	$A$-module and a left Hilbert $B$-module such that
	\begin{enumerate}
		\item[(a)] $\left\langle{x},{y}\right\rangle_B z = x\left\langle{y},{z}\right\rangle_A$ for all $x,y,z\in X$;
		\item[(b)] $\left\langle{X},{X}\right\rangle_A$ spans a dense subset of~$A$ and $\left\langle{X},{X}\right\rangle_B$ spans a dense
		subset of~$B$.
	\end{enumerate}
	We call $A$ and~$B$ \emph{strongly Morita equivalent} if there is an
	$A$-$B$-equivalence bimodule.
\end{definition}

\subsection{Noncommutative finite-fold coverings}

\begin{definition}
	If $A$ is a $C^*$- algebra then an action of a group $G$ is said to be {\it involutive } if $ga^* = \left(ga\right)^*$ for any $a \in A$ and $g\in G$. The action is said to be \textit{non-degenerated} if for any nontrivial $g \in G$ there is $a \in A$ such that $ga\neq a$. 
\end{definition}
\begin{definition}\label{fin_def_uni}
	Let $A \hookto \widetilde{A}$ be an injective *-homomorphism of unital $C^*$-algebras. Suppose that there is a non-degenerated involutive action $G \times \widetilde{A} \to \widetilde{A}$ of a finite group $G$, such that $A = \widetilde{A}^G\stackrel{\text{def}}{=}\left\{a\in \widetilde{A}~|~ a = g a;~ \forall g \in G\right\}$. There is an $A$-valued product on $\widetilde{A}$ given by
	\begin{equation}\label{finite_hilb_mod_prod_eqn}
	\left\langle a, b \right\rangle_{\widetilde{A}}=\sum_{g \in G} g\left( a^* b\right) 
	\end{equation}
	and $\widetilde{A}$ is an $A$-Hilbert module. We say that a triple $\left(A, \widetilde{A}, G \right)$ is an \textit{unital noncommutative finite-fold  covering} if $\widetilde{A}$ is a finitely generated projective $A$-Hilbert module.
\end{definition}
\begin{remark}
	Above definition is motivated by the Theorem \ref{pavlov_troisky_thm}.
\end{remark}
\begin{theorem}\label{pavlov_troisky_thm}\cite{pavlov_troisky:cov}. 
	Suppose $\mathcal X$ and $\mathcal Y$ are compact Hausdorff connected spaces and $p :\mathcal  Y \to \mathcal X$
	is a continuous surjection. If $C(\mathcal Y )$ is a projective finitely generated Hilbert module over
	$C(\mathcal X)$ with respect to the action
	\begin{equation*}
	(f\xi)(y) = f(y)\xi(p(y)), ~ f \in  C(\mathcal Y ), ~ \xi \in  C(\mathcal X),
	\end{equation*}
	then $p$ is a finite-fold  covering.
\end{theorem} 
\begin{definition}\label{fin_comp_def}
	Let $A$, $\widetilde{A}$ be $C^*$-algebras such  that following conditions hold:
	\begin{enumerate}
		\item[(a)] 
		There are unital $C^*$-algebras $B$, $\widetilde{B}$  and inclusions 
		$A \subset B$,  $\widetilde{A}\subset \widetilde{B}$ such that $A$ (resp. $B$) is an essential ideal of $\widetilde{A}$ (resp. $\widetilde{B}$),
		\item[(b)] There is an unital  noncommutative finite-fold covering $\left(B ,\widetilde{B}, G \right)$,
		\item[(c)] 
		\begin{equation}\label{wta_eqn}
		\widetilde{A} =  \left\{a\in \widetilde{B}  ~|~ \left\langle \widetilde{B} ,a  \right\rangle_{\widetilde{B} } \in A \right\}.
		\end{equation}
	\end{enumerate}
	
	The triple $\left(A, \widetilde{A},G \right)$ is said to be a \textit{noncommutative finite-fold covering with compactification}. The group $G$ is said to be the \textit{covering transformation group} (of $\left(A, \widetilde{A},G \right)$ ) and we use the following notation
	\begin{equation}\label{group_cov_eqn}
	G\left(\widetilde{A}~|~A \right) \stackrel{\mathrm{def}}{=} G.
	\end{equation}
\end{definition}

\begin{remark}
Any unital noncommutative finite-fold covering is a noncommutative finite-fold covering with compactification.
\end{remark}

\begin{definition}\label{fin_def}
	Let $A$, $\widetilde{A}$ be $C^*$-algebras, $A\hookto\widetilde{A}$ an injective *-homomorphism and $G\times \widetilde{A}\to \widetilde{A}$ an involutive non-degenerated action of a finite group $G$  such  that following conditions hold:
	\begin{enumerate}
		\item[(a)] 
		$A \cong \widetilde{A}^G \stackrel{\mathrm{def}}{=} \left\{a\in \widetilde{A}  ~|~ Ga = a \right\}$,
		\item[(b)] 
		There is a family $\left\{I_\la \subset A \right\}_{\la \in \La}$ of closed ideals of $A$ such that $\bigcup_{\la \in \La} I_\la$ is a dense subset of $A$ and for any $\la \in \La$ there is a natural  noncommutative finite-fold covering with compactification $\left(I_\la, I_\la \widetilde{A} I_\la, G \right)$.  
	\end{enumerate}
We say that the triple  $\left(A, \widetilde{A},G \right)$ is a \textit{noncommutative finite-fold covering}.
\end{definition}

\begin{remark}
	Any noncommutative finite-fold covering with compactification is a  noncommutative finite-fold covering.
\end{remark}
\begin{remark}
	The Definition \ref{fin_def} is motivated by the Theorem \ref{comm_fin_thm}.
\end{remark}


\begin{thm}\label{comm_fin_thm}\cite{ivankov:qnc}
	If $\mathcal X$, $\widetilde{\mathcal X}$  are  locally compact spaces, and  $\pi: \widetilde{\mathcal X}\to \mathcal X$ is a surjective continuous map, then following conditions are equivalent:
	\begin{enumerate}
		\item [(i)] The map $\pi: \widetilde{\mathcal X}\to \mathcal X$ is a finite-fold regular covering,
		\item[(ii)] There is the natural  noncommutative finite-fold covering $\left(C_0\left(\mathcal  X \right), C_0\left(\widetilde{\mathcal X} \right), G    \right)$.
	\end{enumerate}
\end{thm}

\section{Main Result}

\paragraph{}
The following result about algebraic Morita equivalence of Galois extensions is in fact rephrasing of described in \cite{auslander:galois,mont:hopf-morita} constructions. 
Let $\left(A, \widetilde{A}, G\right)$ be an unital finite-fold noncommutative covering. Denote by $\widetilde{A} \rtimes G$ a crossed product, i.e. $\widetilde{A} \rtimes G$ ia an algebra which coincides with a set of maps from $G$ to $\widetilde{A}$ as a set, and operations on  $\widetilde{A}$ are given by
\begin{equation*}
\begin{split}
\left(a + b \right) \left( g \right)   = a \left( g \right)+ b \left( g \right),\\
\left(a \cdot b \right) \left( g \right)   = \sum_{g'\in G} a\left( g' \right)\left( g'\left(  b\left(g'^{-1}g \right)\right)\right)  , \\ \forall a, b \in \widetilde{A} \rtimes G,~~ \forall g \in G.
\end{split}
\end{equation*}
If an involution on $\widetilde{A} \rtimes G$ is given by
$$
a^*\left(g \right) =\left(a\left( g^{-1}\right)  \right)^*; ~~ \forall a \in \widetilde{A} \rtimes G,~~ \forall g \in G,
$$
then $\widetilde{A} \rtimes G$ is a $C^*$-algebra.
Let us construct a Morita context $$\left( \widetilde{A} \rtimes G, A,~ _{\widetilde{A} \rtimes G}\widetilde{A}_A, ~ _A\widetilde{A}_{\widetilde{A} \rtimes G}, \varphi, \psi\right) $$ where both $_A\widetilde{A}_{\widetilde{A} \rtimes G}$ and $ _{\widetilde{A} \rtimes G}\widetilde{A}_A$ coincide with $\widetilde{A}$ as $\C$-spaces. Left and right action of $\widetilde{A} \rtimes G$ on $\widetilde{A}$ is given by

\begin{equation*}
\begin{split}
a \widetilde{a} = \sum_{g \in G} a\left( g\right) \left( g\widetilde{a}\right)   , \\
\widetilde{a} a =\sum_{g \in G} g^{-1}\left(\widetilde{a} a\left( g\right)\right),\\ 
\forall a \in \widetilde{A} \rtimes G,~~ \forall \widetilde{a}\in \widetilde{A}.
\end{split}
\end{equation*}
Left (resp. right) action of $A$ on $\widetilde{A}$ we define as left (resp. right) multiplication by $A$.
Denote by $\varphi: \widetilde{A} \otimes_{A}  \widetilde{A}  \to \widetilde{A} \rtimes G$, $\psi: \widetilde{A} \otimes_{\widetilde{A} \rtimes G}  \widetilde{A} \to A$ the maps such that
\begin{equation*}
\begin{split}
\varphi \left(\widetilde{a} \otimes \widetilde{b} \right)\left( g\right) =    \widetilde{a} \left(g\widetilde{b} \right) , \\
\psi \left(\widetilde{a} \otimes \widetilde{b} \right) = \sum_{g \in G} g\left(\widetilde{a} \widetilde{b} \right), \\
 \forall \widetilde{a} ,\widetilde{b}\in \widetilde{A}, ~~ g \in G.
\end{split}
\end{equation*}
From above equations it follows that 
\begin{equation}\label{morita_g_eqn}
\begin{split}
\varphi \left(\widetilde{a} \otimes \widetilde{b} \right)\widetilde{c} = \sum_{g \in G} \left( \varphi\left(\widetilde{a} \otimes \widetilde{b} \right)\left( g\right)\right)  g \widetilde{c} = \sum_{g \in \G} \widetilde{a} \left( g \widetilde{b} \right) \left( g \widetilde{c}\right) =  \widetilde{a}\sum_{g \in G}  g\left(  \widetilde{b}  \widetilde{c}\right) = \widetilde{a} \psi\left(\widetilde{b}\otimes \widetilde{c} \right),\\
\widetilde{a}\varphi \left(\widetilde{b} \otimes \widetilde{c} \right) =
\sum_{g \in G} g^{-1}\left( \widetilde{a} \widetilde{b} g \widetilde{c}\right) = \left( \sum_{g \in G} g^{-1}\left( \widetilde{a} \widetilde{b} \right) \right)\widetilde{c}= \left( \sum_{g \in G} g\left( \widetilde{a} \widetilde{b} \right) \right)\widetilde{c}=\psi\left(\widetilde{a} \otimes \widetilde{b} \right)\widetilde{c},
\end{split}
\end{equation}
i.e. $\varphi, \psi$ satisfy conditions \eqref{morita_ctx_eqn}, so $\left( \widetilde{A} \rtimes G, A,~ _{\widetilde{A} \rtimes G}\widetilde{A}_A, ~ _A\widetilde{A}_{\widetilde{A} \rtimes G}, \varphi, \psi\right) $ is a Morita context.  Taking into account that the $A$-module $\widetilde{A}_A$ is a finitely generated projective generator and Remark \ref{morita_rem}  one has a following lemma.
\begin{lem}\label{morita_galois_lem}
	If $\left(A, \widetilde{A}, G \right)$ is an unital noncommutative finite-fold  covering then  $$\left( \widetilde{A} \rtimes G, A,~ _{\widetilde{A} \rtimes G}\widetilde{A}_A, ~ _A\widetilde{A}_{\widetilde{A} \rtimes G}, \varphi, \psi\right)$$ is an algebraic Morita equivalence. 
\end{lem}

\begin{cor}\label{unital_cov_cor}
Let $\left(A, \widetilde{A}, G \right)$ be an unital noncommutative finite-fold  covering. Let us define a structure of Hilbert $\widetilde{A} \rtimes G-A$ bimodule on  $_{\widetilde{A} \rtimes G}\widetilde{A}_A$ given by following products
\begin{equation}\label{hilb_gal_eqn}
\begin{split}
\left\langle a, b \right\rangle_{\widetilde{A} \rtimes G} = \varphi\left(a \otimes b^* \right),\\ 
\left\langle a, b \right\rangle_A = \psi\left( a^* \otimes b\right).\\
\end{split}
\end{equation}
Following conditions hold:
	\begin{enumerate}
		\item [(i)]
A bimodule  $_{\widetilde{A} \rtimes G}\widetilde{A}_A$ satisfies the associativity condition (a) of the Definition \ref{strong_morita_defn},

\item[(ii)] 
\begin{equation*}
\begin{split}
\left\langle \widetilde{A} , \widetilde{A} \right\rangle_{\widetilde{A} \rtimes G} = \widetilde{A} \rtimes G,\\ 
\left\langle \widetilde{A} , \widetilde{A}  \right\rangle_A = A.\\
\end{split}
\end{equation*}
	\end{enumerate}
It follows that $_{\widetilde{A} \rtimes G}\widetilde{A}_A$  is a ${\widetilde{A} \rtimes G}-A$  equivalence bimodule.
\end{cor}
\begin{proof}
(i) From \eqref{morita_g_eqn} it follows that products $ \left\langle -, - \right\rangle_{\widetilde{A} \rtimes G}$, $~~\left\langle -, - \right\rangle_A$ satisfy condition (a) of the Definition \ref{strong_morita_defn}.\\
(ii) From the Lemma  \ref{morita_galois_lem} and the definition of algebraic Morita equivalence  it turns out
\begin{equation*}
\begin{split}
\varphi\left(  \widetilde{A} \otimes_{A}  \widetilde{A}\right)  =\left\langle \widetilde{A} , \widetilde{A} \right\rangle_{\widetilde{A} \rtimes G} = \widetilde{A} \rtimes G,\\
\psi\left( \widetilde{A} \otimes_{\widetilde{A} \rtimes G}  \widetilde{A}\right)=\left\langle \widetilde{A} , \widetilde{A} \right\rangle_{A}  = A.
\end{split}
\end{equation*}
\end{proof}
Let us consider the situation of the Corollary \ref{unital_cov_cor}. Denote by $e \in G$ the neutral element. The unity $1_{\widetilde{A} \rtimes G}$ of $\widetilde{A} \rtimes G$ is given by
\begin{equation}
\begin{split}
1_{\widetilde{A} \rtimes G}\left(g \right) = \left\{\begin{array}{c l}
1_{\widetilde{A}}  & g = e \\
0 & g \neq e
\end{array}\right..
\end{split}
\end{equation}
From the Lemma \ref{morita_galois_lem} it follows that there are $\widetilde{a}_1,...,\widetilde{a}_n, \widetilde{b}_1,..., \widetilde{b}_n \in \widetilde{A}$ such that
\begin{equation}\label{1_g_eqn}
\begin{split}
1_{\widetilde{A} \rtimes G} = \varphi\left( \sum_{j=1}^n \widetilde{a}_j \otimes\widetilde{ b}_j^*\right) = \sum_{j=1}^n \left\langle \widetilde{a}_j , \widetilde{b}_j \right\rangle_{\widetilde{A} \rtimes G}.
\end{split}
\end{equation}
From the above equation it turns out that for any $g \in G$
\begin{equation}\label{ga_eqn}
\varphi\left( \sum_{j=1}^n g\widetilde{a}_j \otimes \widetilde{b}_j^*\right)\left(g' \right) = \left(\sum_{j=1}^n \left\langle g\widetilde{a}_j , \widetilde{b}_j \right\rangle_{\widetilde{A} \rtimes G} \right) \left(g' \right) \left\{\begin{array}{c l}
1_{\widetilde{A}}  & g' = g \\
0 & g' \neq g
\end{array}\right..
\end{equation}

\begin{lem}\label{comp_cov_lem}
	Let $\left(A, \widetilde{A}, G \right)$ be a noncommutative finite-fold  covering with compactification. Let us define a structure of Hilbert $\widetilde{A} \rtimes G-A$ bimodule on  $_{\widetilde{A} \rtimes G}\widetilde{A}_A$ given by  products \eqref{hilb_gal_eqn}
	Following conditions hold:
	\begin{enumerate}
		\item [(i)]
		A bimodule  $_{\widetilde{A} \rtimes G}\widetilde{A}_A$ satisfies associativity conditions (a) of the Definition \ref{strong_morita_defn}
		\item[(ii)] 
		\begin{equation*}
		\begin{split}
		\left\langle \widetilde{A} , \widetilde{A}  \right\rangle_A = A,\\
				\left\langle \widetilde{A} , \widetilde{A} \right\rangle_{\widetilde{A} \rtimes G} = \widetilde{A} \rtimes G.\\ 
		\end{split}
		\end{equation*}
	\end{enumerate}
	It follows that $_{\widetilde{A} \rtimes G}\widetilde{A}_A$  is a ${\widetilde{A} \rtimes G}-A$  equivalence bimodule.
\end{lem}
\begin{proof}(i)
From the Definition \ref{fin_comp_def} it follows that there are  unital $C^*$-algebras $B$, $\widetilde{B}$  and inclusions 
$A \subset B$,  $\widetilde{A}\subset \widetilde{B}$ such that $A$ (resp. $B$) is an essential ideal of $\widetilde{A}$ (resp. $\widetilde{B}$). Moreover there is an unital  noncommutative finite-fold covering $\left(B ,\widetilde{B}, G \right)$. From the Corollary \ref{unital_cov_cor} it turns out that a bimodule  $_{\widetilde{B} \rtimes G}\widetilde{B}_B$ is a $\widetilde{B} \rtimes G-B$ equivalence bimodule. Both scalar products $ \left\langle -, - \right\rangle_{\widetilde{A} \rtimes G}$, $~~\left\langle -, - \right\rangle_A$ are restrictions of products $ \left\langle -, - \right\rangle_{\widetilde{B} \rtimes G}$, $~~\left\langle -, - \right\rangle_B$, so products $ \left\langle -, - \right\rangle_{\widetilde{A} \rtimes G}$, $~~\left\langle -, - \right\rangle_A$ satisfy to condition (a) of the Definition \ref{strong_morita_defn}.\\
(ii) From \eqref{1_g_eqn} it turns out that there are $\widetilde{a}_1,...,\widetilde{a}_n, \widetilde{b}_1,..., \widetilde{b}_n \in \widetilde{B}$ such that
\begin{equation*}
\begin{split}
1_{\widetilde{B} \rtimes G} = \varphi\left( \sum_{j=1}^n \widetilde{a}_j \otimes\widetilde{ b}_j^*\right) = \sum_{j=1}^n \left\langle \widetilde{a}_j , \widetilde{b}_j \right\rangle_{\widetilde{B} \rtimes G}.
\end{split}
\end{equation*}

If $a \in A_+$ is a positive element then there is $x \in A \subset \widetilde{A}$ such that $a = x^*x$ it follows that
$$
a = \frac{1}{\left|G\right|}\left\langle x, x \right\rangle_A.
$$
Otherwise $A$ is the $\C$-linear span of positive elements it turns out
$$
\left\langle \widetilde{A} , \widetilde{A}  \right\rangle_A = A.
$$

For any positive  $\widetilde{a} \in \widetilde{A}_+$ and any $g \in G$ denote by
 $y^{\widetilde{a}}_g \in \widetilde{A} \rtimes G $ given by
\begin{equation*}
y^{\widetilde{a}}_g\left(g' \right) \left\{\begin{array}{c l}
\widetilde{a}  & g' = g \\
0 & g' \neq g
\end{array}\right..
\end{equation*}
 There is $\widetilde{x} \in \widetilde{A}$ such that $\widetilde{x}\widetilde{x}^*= \widetilde{a}$. Clearly $\widetilde{x}\left( g\widetilde{a}_j\right) , \widetilde{x}\widetilde{b}_j \in \widetilde{A}$ and from \eqref{ga_eqn} it follows that
$$
y^{\widetilde{a}}_g = \sum_{j=1}^n \left\langle x \left( g\widetilde{a}_j\right)  , x\widetilde{b}_j \right\rangle_{\widetilde{A} \rtimes G}
$$
The algebra $\widetilde{A} \rtimes G$ is the $\C$-linear span of elements $y^{\widetilde{a}}_g$, so one has $\left\langle \widetilde{A} , \widetilde{A} \right\rangle_{\widetilde{A} \rtimes G} = \widetilde{A} \rtimes G$. 
	
\end{proof}

\begin{thm}\label{main_theorem}
	If $\left(A, \widetilde{A}, G \right)$ noncommutative finite-fold covering  then a Hilbert $\left( \widetilde{A} \rtimes G,A\right) $  bimodule  $_{\widetilde{A} \rtimes G}\widetilde{A}_A$ is a $\widetilde{A} \rtimes G-A$ equivalence bimodule.
\end{thm}
\begin{proof}
	From the Definition \ref{fin_def} there is a family $\left\{I_\la \subset A \right\}_{\la \in \La}$ of closed ideals of $A$ such that $\bigcup_{\la \in \La} I_\la$ is a dense subset of $A$ and for any $\la \in \La$ there is a natural  noncommutative finite-fold covering with compactification $\left(I_\la, I_\la \widetilde{A} I_\la, G \right)$. From the Lemma \ref{comp_cov_lem} it turns out that   $I_\la\widetilde{A}I_\la$ is a $I_\la\widetilde{A}I_\la \rtimes G-I_\la$ equivalence bimodule and
		\begin{equation}\label{ia_eqn}
\begin{split}
\left\langle  I_\la\widetilde{A} I_\la ,  I_\la\widetilde{A}  I_\la \right\rangle_{ I_\la\widetilde{A} I_\la \rtimes G} =  I_\la\widetilde{A} I_\la \rtimes G,\\ 
\left\langle  I_\la\widetilde{A} I_\la ,  I_\la\widetilde{A}   I_\la\right\rangle_{ I_\la} =  I_\la.\\
\end{split}
\end{equation}
	for any $\la\in \La$.
The union $\bigcup_{\la \in \La} I_\la$ is a dense subset of $A$ and $\bigcup_{\la \in \La} I_\la \widetilde{A} I_\la$ is a dense subset of $\widetilde{A}$. So domains of products  $\left\langle -, - \right\rangle_{I_\la\widetilde{A} I_\la \rtimes G}~$, $~~\left\langle -, - \right\rangle_{I_\la}$ can be extended up to $\widetilde{A} \times \widetilde{A}$ and resulting products satisfy to  (a) of the Definition \ref{strong_morita_defn}. From \eqref{ia_eqn} it turns out that $\left\langle \widetilde{A}  ,  \widetilde{A}  \right\rangle_{ \widetilde{A}  \rtimes G}$ (resp. $\left\langle  \widetilde{A}  ,  \widetilde{A}  \right\rangle_{ A}$) is a dense subset of  $\widetilde{A} \rtimes G$ (resp. $A$), i.e. the condition (b) of the Definition \ref{strong_morita_defn} holds.
\end{proof}

\end{document}